\documentclass[10pt,a4paper,reqno]{amsart}

\usepackage{amsmath}
\usepackage{amssymb}
\usepackage{amsthm}
\usepackage{amscd}
\usepackage{xcolor}
\usepackage{amsaddr}
\usepackage{fancyhdr}
\usepackage[draft=true]{hyperref}
\usepackage{cleveref}
\usepackage[bottom=3cm,top=3.5cm, left=3cm, right=3cm]{geometry}
\usepackage[normalem]{ulem}
\usepackage{amsfonts}
\usepackage{amsbsy}
\usepackage[utf8]{inputenc}
\usepackage[shortlabels]{enumitem}
\usepackage{cancel}
\usepackage{xspace}
\usepackage[greek,english]{babel}
\usepackage{url}
\usepackage{wrapfig}
\usepackage{enumitem}
\usepackage{multicol}
\usepackage{moreenum}
\usepackage{xcolor}
\usepackage{subcaption}
\usepackage{tikz}
\usepackage{array,multirow}
\usepackage[all,cmtip]{xy}
\definecolor{LinkColor}{rgb}{0,0,0} 

\newtheorem*{Main Theorem} {Main Theorem}
\newtheorem*{Theorem A} {Theorem A}
\newtheorem*{Theorem B} {Theorem B}
\newtheorem*{Theorem C} {Theorem C}
\newtheorem*{Conjecture A}{Conjecture A}
\newtheorem*{Conjecture B}{Conjecture B}

\newtheorem*{Question B}   {Question B}

\newtheorem {theorem}    {Theorem}[section]    
\newtheorem {lemma}      [theorem]    {Lemma}

\newtheorem {proposition}[theorem]    {Proposition}

\newtheorem {remark} 	 [theorem]    {Remark}

\numberwithin{equation}{section}

\newcounter{DM@bibnum}

\DeclareMathOperator{\Ind}{Ind}
\DeclareMathOperator{\Res}{Res}

\DeclareMathOperator{\SL}{SL}

\DeclareMathOperator{\Irr}{Irr}

\newcommand{\Syl}{\operatorname{Syl}}

\newcommand{\F}{\mathbb{F}}

\newcommand{\NN}{\mathbb{N}}
\newcommand{\ZZ}{\mathbb{Z}}
\newcommand{\QQ}{{\mathbb Q}}

\newcommand{\GEN}[1]{\left\langle #1 \right\rangle}

\newcommand{\qand}{\quad \text{and} \quad}
\newcommand{\cut}{\textsf{cut}\xspace}

\newcommand{\MM}{\mathbb{M}}

\title{A generalization of a result of Hegedüs}

\author{Sara C. Debón}
\address{Departamento de Matem\'aticas, Universidad de Murcia, Spain}
\email{sara.cebelland@um.es}

\begin{document}
	\maketitle
	
	\section{Introduction}
	Recall that a finite group $G$ is said to be \textit{rational} if for every $g \in G$ every generator of $\GEN{g}$ is conjugated to $g$. The structure of finite rational groups has been widely studied in, for instance \cite{FeitSeitz1989, IsaacNav, Kletzing1984, Thompson2008}. In particular, the main result of \cite{Hegedus2005} shows that in a finite solvable rational group, a Sylow
	$5$-subgroup is normal and elementary abelian. A key ingredient in the proof of this result is the following theorem where $\MM$ is the Frobenius group $C_{5}^2 \rtimes Q_8$ with $Q_8$ acting on $C_5^2$ as a subgroup of $\SL(2,5)$.
	
	\medskip
	\begin{theorem}\label{Theo1.1} \textup{(Theorem 1.2 of \cite{Hegedus2005})} If $H$ is a finite rational $\{2, 5\}$-group, then $P\in \Syl_5(H)$ is normal and $P\cong C_5^{2n}$ for some $n\in \NN$. Moreover, $H/O_2(H)\cong \MM\wr K$ for $K$ a $2$-subgroup of $\Sigma_n$.
	\end{theorem}
	
	In fact, $K$ satisfies some technical property which we omit here.
	
	These results of Hegedüs show that every finite solvable rational group is of the form $V\rtimes G$ with $V$ an
	elementary abelian Sylow $5$-subgroup and $G$ a rational group acting on $V$ with the \textit{eigenvector property}, that is, for every $v \in V$ and $\lambda \in \F_5\setminus \{0\}$ there exists some $g \in G$ such that $vg=\lambda v$. This situation of a
	group $G$ acting on an elementary abelian $p$-group $V$ appeared several times in \cite{BKMdR}, \cite{DebonGLucasdelRio24}+\cite{DebonGLucasdelRio25}, \cite{Gow} and \cite{Hegedus2005}, without $V \rtimes G$ being rational. In this paper we show how the ideas in Sections 2 and 3 of Hegedüs’ paper that prove Theorem \ref{Theo1.1} can be adapted to this more general situation leading to the following
	result.
	
	\begin{Theorem A}\label{TheoremA}
		Let $p \geq 5$ be prime, $G$ a finite rational  $2$-group and $V$ a non-zero finite dimensional vector
		space over $\F_p$ on which $G$ acts faithfully and with the eigenvector property. Then $p = 5$ and there is a
		$2$-subgroup $K$ of the symmetric group $\Sigma_n$ such that $G\cong Q_8\wr K$, $\dim V = 2n$ and $V \rtimes G \cong \MM \wr K$.
	\end{Theorem A}
	
	We discovered Theorem A studying the Gruenberg-Kegel graphs of finite solvable \cut groups encountering situations like the ones described above where the groups involved might not be rational.
	In a work in progress, \cite{DebonGLucasdelRio26}, we will present some applications of the theorem that may help to exclude some of the graphs in Question E of \cite{BKMdR}.
	
	\section{Notation and Preliminaries}
	
	We collect here some notation and definitions used in the next sections. Moreover, we include some lemmas about non-singular bilinear forms that are necessary to prove the main theorem.
	
	Let $G$ be a finite group, $g \in G$, $X \subseteq G$ and $H \leq G$. We denote by $|G|$ the order of $G$, $|g|$ the order of $g$,
	$|G : H|$ the index of $H$ in $G$, $C_G(X)$ the centralizer of $X$ in $G$ and $N_G(H)$ the normalizer of $H$ in $G$.
	We recall that $G$ is \textit{rational} if every generator of $\GEN{g}$ is conjugated to $g$, for every $g \in G$. Also, $\times$, $\rtimes$
	and $\wr$ refer to the direct, semi-direct and wreath products, respectively.
	
	Let $p$ be a prime and $n$ a positive integer. Then $\F_{p^n}$ is the field with $p^n$ elements, $\Sigma_n$ is the
	symmetric group on $n$ letters, $Q_{2^n}$ is the generalized quaternion group of order $2^n$ and $C_n$ is the cyclic group of order $n$.
	
	All modules will be assumed to be on the right. Let $\F$ be a field, $V$ an $\F G$-module of finite dimension, $A \subseteq V$ and $W \leq V$. Then, $\dim V$ is the dimension of $V$ (over $\F$) and similarly to the group theoretical definitions of $C_G$ and $N_G$ we denote
	
	$$N_G(W) = \{g \in G : Wg = W\},$$
	$$C_G(A) = \{g \in G : ag = a \text{ for all } a \in A\}$$ 
	and
	$$C_V(X) = \{v \in V : vx = v \text{ for all } x \in X\}.$$
	
	Moreover, if $\alpha \in \F \setminus \{0\}$, then $V (g - \alpha) = \{v(g - \alpha) : v \in V \}$ and $V_g(\alpha) = \{v \in V : vg = \alpha v\}$, that is, the image and kernel of the endomorphism $(g - \alpha \cdot id)$ on $V$. We say that $G$ acts on $V$ with the
	\textit{eigenvector property} or that $V$ has the \textit{eigenvector property} if for every $v \in V$ and $\lambda \in \F \setminus \{0\}$ there
	exists some $g \in G$ with $v \in V_g(\lambda)$.
	
	If $N$ is a normal $p$-elementary abelian subgroup of $G$ then $N$ can be regarded as a $\F_pG$-module (using the action of $G$ on $N$ by conjugation).
	In particular, $N$ is also a $\F_p(G/N)$-module. Conversely, if $V$ is an $\F_pG$-module, $V \rtimes G$ refers to the
	semi-direct product associated with the action of $G$ on $V$. If $G$ acts on $V$ fixed-point-freely then $V\rtimes G$
	is Frobenius and we express this by $V\rtimes_{Fr} G$. We recall that in a Frobenius group, the Frobenius
	complement has cyclic or generalized quaternion Sylow subgroups.
	
	If $W$ is an $\F H$-module then $\Ind_H^G(W)$ refers to the induction module of $W$ to $G$, i.e., $W \otimes_{\F H} \F G$.
	If $H$ is clear from the context we just use $W^G$. On the other hand, the restriction of $V$ to $H$ will be
	denoted by $\Res_H(V)$ or $V_H$.
	
	Note that there is some connection between the concept of rationality and the eigenvector property.
	Indeed, the usual situation is that $G$ is some solvable rational group, and hence, a minimal normal $p$-subgroup of $G$ is an $\F_pG$-module having the eigenvector property.
	
	From now until the end of the section we focus our attention on bilinear forms. Let $(-,-)$ be a bilinear form on $V$. We recall that $(-,-)$ is called:
	
	\begin{itemize}
	\item \textit{non-singular}, if having $(v,w) = 0$ for every $w \in V$ implies $v = 0$ (and the same occurs if we
	interchange the order of $v$ and $w$);
	\item $G$-\textit{invariant}, if $(vg,wg) = (v,w)$ for every $v,w \in V$ and $g \in G$;
	\item \textit{alternating}, if $(v, v) = 0$ for every $v \in V$; and
	\item \textit{symplectic}, if $(-,-)$ is non-singular and alternating.
	\end{itemize}
	
	If $(-,-)$ is alternating, then it easily follows that for $v,w \in V$ then $(v,w) = 0$ if and only if $(w, v) = 0$. In that case, if $W$ is a subspace of $V$, then we denote by $W^\perp$ the orthogonal of $W$, i.e., $$W^\perp = \{v \in V : (v,w) = 0 \text{ for every } w \in W\}.$$
	Hence, for symplectic $(-,-)$, $\dim V = \dim W + \dim W^\perp$ and $(W^\perp)^\perp = W$.
	
	The idea for the next remark is taken from the proof of Theorem 1 of \cite{Gow}.
	
	\begin{remark}\label{eigvecsymplectic} \textup{Let $\F$ be a field of cardinality greater than $3$, let $G$ be a group and let $V$ be an
	$\F G$-module with the eigenvector property. Then every $G$-invariant non-singular form $(-,-)$ of $V$ is symplectic. Indeed, as $|\F| > 3$, $\F \setminus \{0\}$ contains an element $\alpha$ such that $\alpha^2\neq 1$. Let $v \in V$.
	By the eigenvector property, there is $g \in G$ such that $vg = \alpha v$. Now since the form is $G$-invariant,
	$(v, v) = (vg, vg) = (v, v)\alpha^2$ and hence $(v, v) = 0$.}
	\end{remark}
	
	Finally, Lemma \ref{nonsingular} is based on Proposition 2.3 of \cite{Hegedus2005}, which at the same time follows Step 1 in Lemma 6 of \cite{Gow}.
	
	\begin{lemma}\label{nonsingular} Let $\F$ be a field, let $G$ be a group and let $V$ be a non-zero $\F G$-module. Let $C$ be a maximal element of $\{C_G(v) : v \in V \setminus \{0\}\}$ and set $N = N_G(C)$ and $W = C_V (C)$. Then:
		\begin{enumerate}
	\item $C = C_G(W) = C_G(w)$ for every $w \in W \setminus \{0\}$,
	\item $N = N_G(W)$,
	\item If $V$ has the eigenvector property, then so does $W$ as $\F N$-module; and 
	\item If the characteristic of $\F$ does not divide the order of $G$, then the restriction to $W$ of a $G$-invariant non-singular form of $V$ is also non-singular.
	\end{enumerate}
	\end{lemma}
	\begin{proof} (1) Note that $W\neq 0$. If $w\in W\setminus \{0\}$, then $C\subseteq  C_G(w)$, and by the maximality of $C$, the equality holds.
	Then $C_G(W) = \bigcap_{w\in W\setminus \{0\} } C_G(w)=C$.

	(2) Let $n \in N$ and $w \in W$. Then for every $c \in C$, $wncn^{-1} = w$, so that $wnc = wn$. Therefore $wn \in C_V (C)=W$. This shows that $n \in N_G(W)$. Conversely, let $g \in N_G(W)$ and let $c \in C$. Then for every $w \in W$, $wg \in W$ and therefore $wgc = wg$. Thus $gcg^{-1} \in C_G(W) = C$. This proves that $g \in N_G(C) = N$.
	
	(3) Suppose now that $V_G$ has the eigenvector property. Let $\alpha \in \F \setminus \{0\}$ and $w \in W \setminus \{0\}$. Then there is $g \in G$ such that $wg = \alpha w$. Then $wg \in W = C_V (C)$ and so $wgc = wg$ for every $c \in C$. Then $gcg^{-1} \in C_G(w) = C$. This shows that $g \in N_G(C) = N$, and hence $W$ has the eigenvector property as $\F N$-module.
	
	(4) Let $(-,-)$ be a non-singular $G$-invariant form of $V$. For $X$ and $Y$ subsets of $V$ denote $$l_X(Y) = \{x \in X : (x, y) = 0  \text{ for all } y \in Y \} \qand r_X(Y) = \{x \in X : (y, x) = 0 \text{ for all } y \in Y \}.$$
	The non-singularity of $V$ means that $l_V(V) = r_V(V) = 0$. We have to show that $l_W(W) = r_W(W) = 0$.
	By symmetry we only prove $l_W(W) = 0$. Let $W_0 = l_W(W)$. Since $(-,-)$ is $G$-invariant and $W$ is an $\F N$-submodule of $V$, then so are $W_0$ and $r_V (W_0)$. By semi-simplicity, $V = r_V (W_0)\oplus U$ for some $\F N$-submodule $U$ of $V$. Thus $U \cong V/r_V (W_0)$ as $\F N$-modules. 
	
	We claim that $C \subseteq C_N(U)$, or equivalently that $C \subseteq C_N(V/r_V(W_0))$. Indeed, let $c \in C$. As $W_0 \subseteq W = C_V (C)$, for every $w_0 \in W_0$, $w_0c^{-1} = w_0$, and hence for every $v \in V$, $(w_0, vc) = (w_0c^{-1}, v) = (w_0, v)$. Therefore $vc-v \in r_V (W_0)$ for every $v \in V$ , i.e. $c \in C_N(V/r_V (W_0))$. This finishes the proof of the claim.
	
	Since $C \subseteq C_N(U)$ we have $U + W \subseteq C_V(C) = W$, so that $U \subseteq W \subseteq r_V (l_W(W)) = r_V (W_0)$ and hence $U = 0$. Therefore, $V = r_V (W_0)$, and hence, $W_0 \subseteq l_V(V) = 0$, as desired.
	\end{proof}
	
	\section{Proof of Theorem A}
	The aim of this section is to prove Theorem A using the proof of Theorem 3.6 of \cite{Hegedus2005}. We start with a remark and two propositions which are direct
	generalizations of Proposition 2.3 and Proposition 3.1 of \cite{Hegedus2005}. Their proofs are almost exactly the same.
	
	\begin{remark}\label{degree2m-2}\textup{ 
		Let $S=Q_{2^m}$. A faithful and simple $\F_5S$-module has degree $2^{m-2}$. Indeed, write $d=2^{m-1}$, let $g\in S$ have order $d$ and let $\mathfrak{X}$ be a faithful absolutely irreducible $\F_5$-representation of $S$. By 3.1 and 3.4 of \cite{Q8degree},
		$$
		\mathfrak{X}(g)=
			\begin{pmatrix}
			\xi_{d}&0\\0&\xi_d^{-1} 
			\end{pmatrix},
		$$
		where $\xi_d$ is some root of unity of order $d$ in some extension of $\F_5$. Hence, by Theorem 9.21 of \cite{Isaacs1976} a faithful and simple $\F_5S$-module has degree $$2\cdot [\F_5(\xi_d+\xi_d^{-1}):\F_5].$$
		It can be easily proved (by induction on $m$) that $\F_5(\xi_d)=\F_5(\xi_d+\xi_d^{-1})$. Moreover, $[\F_5(\xi_d):\F_5]$ is the order of $5\mod d$. Hence, $$2\cdot [\F_5(\xi_d+\xi_d^{-1}):\F_5]=2^{m-2},$$ as desired.
	}
	\end{remark}
	
	\begin{proposition}\label{p=5}
	Let $p$ and $q$ be different prime integers with $p \geq 5$, let $S$ be a finite $q$-group and let $V$ be a finite dimensional non-zero $\F_pS$-module. Suppose that $V_{\F_pS}$ has the eigenvector property, $V$ has an $S$-invariant non-singular form and the action of $S$ on $V$ is fixed-point free. Then $V \rtimes S \cong \MM$.
	In particular, $p = 5$ and $q = 2$.
	\end{proposition}
	\begin{proof}
	Let $(-,-)$ be a non-singular $S$-invariant form on $V$. By Remark \ref{eigvecsymplectic}, this form is symplectic and the properties introduced in the preliminary section hold.
	
	We claim that $V (s - \alpha) = V_{s^{-1}}(\alpha)^\perp$ for every $s \in S$ and $\alpha \in \F_p \setminus \{0\}$. Indeed, as the form is $S$-invariant, for every $u, v \in V$ we have
		$$(u(s - \alpha), v) = (u, v(s^{-1} - \alpha)).$$
	Therefore, $V (s-\alpha) \subseteq V_{s^{-1}}(\alpha)^\perp$ and if $v \in V (s -\alpha)^\perp$, then $(u, v(s^{-1} - \alpha)) = 0$ for every $u \in V$ and hence $v \in V_{s^{-1}} (\alpha)$, by the non-singularity of the form. The latter shows that $V (s - \alpha)^\perp \subseteq V_{s^{-1}} (\alpha)$ and hence $V_{s^{-1}} (\alpha)^\perp \subseteq V (s - \alpha)^{\perp \perp} = V (s - \alpha)$. This finishes the proof of the claim.
	
	If $\alpha \neq 0$, then $V_s(\alpha) = V_{s^{-1}} (\alpha^{-1}) = V (s - \alpha^{-1})^\perp$ and hence
	$$\dim V_s(\alpha) = \dim V - \dim V (s - \alpha^{-1}) = \dim V_s(\alpha^{-1}).$$
	
	Fix a generator $\alpha$ of the group of units of $\F_p$ and let $J$ denote the set of units of $\ZZ_{p-1}$. 
	Thus
	$\{\alpha^j : j \in J\}$ is the set of elements of order $p - 1$ in $\F_p$.
	
	Let $s \in S$ of order $p - 1$. By semisimplicity, $s$ is diagonalizable over $\F_p$ and, as the action of $S$ on
	$V$ is fixed-point-free, all the eigenvalues of $s$ have order $p - 1$. Hence,
	$$V = \bigoplus_{j\in J} V_s(\alpha^j).$$
	Since $p \geq 5$, if $j \in J$, then $j\not \equiv -j \mod p-1$ and hence $\alpha^j\not = \alpha^{-j}$ . However, $\dim V_s(\alpha^j) = \dim V_s(\alpha^{-j})$ and therefore $\dim V$ is even, say equal to $2d$, and $\dim V_s(\alpha^j)\leq d$ for every $j \in J$.
	
	By the eigenvector property, for every $v \in V \setminus \{0\}$ there is $s \in S$ such that $vs = v\alpha$, i.e. $v \in V_s(\alpha)$.
	Moreover, as the action is fixed-point-free there is a unique $s \in S$ satisfying this condition and $|s| =p-1$. Therefore, $V \setminus \{0\}$ is the disjoint union of the sets of the form $V_s(\alpha) \setminus \{0\}$ with $s$ running on
	the elements of order $p-1$ in $S$. Since $V_s(\alpha)$ has dimension at most $d$, it follows that if $n$ denotes the
	number of elements of $S$ of order $p-1$, then $p^{2d}-1 \leq n(p^d -1)$, and hence $n \geq p^d +1$. In particular,
	$S$ has an element of order $p-1$ and so $p-1 = q^k$ for some integer $k$. This implies that $q = 2$ and $k$ is a
	power of $2$ (by, for instance, Theorem 2.7, IX, \cite{Huppert2}), i.e. $p$ is a Fermat prime other than $3$. Hence, $p = 2^k + 1$.
	
	As $S$ acts fixed-point-freely on $V$, $S$ is either cyclic or quaternion. If $S$ is cyclic, then $n = 2^{k-1} \geq
	(2^k + 1)^d + 1$ which is not possible. Therefore $S = Q_{2^m}$ with $m\geq \max(3, k + 1)$. If $k > 2$, then the
	elements of order $2^k$ of $S$ belong to a cyclic subgroup of order $2^{m-1}$ in $S$ and again $n = 2^{k-1}$, yielding
	a contradiction. Thus $k = 2$, so $p = 5$ and the number of elements of order $4$ in $S$ is $n = 2^{m-1} + 2$.
	
	Using once more that the action of $S$ on $V$ is fixed-point-free, it follows that every simple submodule of $V_{\F_5S}$ is faithful and hence it has degree $2^{m-2}$ by Remark \ref{degree2m-2}. Therefore, if $t$ is the number of summands in the expression of $V_{\F_5S}$ as a direct sum of simple modules, then $2d = 2^{m-2}t$. Thus $2^{m-1} + 2 \geq 5^{2^{m-3}t} + 1$.
	Then $m = 3$ and $t = 1$, so $S = Q_8$ and $d = 1$. Thus $V \rtimes S \cong \MM$.
	\end{proof}
	
	The assumption $p \geq 5$ in Proposition \ref{p=5} is sharp. For example, the unique faithful linear representation of $C_2$ in characteristic $3$ provides an $\F_3C_2$-module with the eigenvector property and fixed-point-free action. The product in $\F_3$ is a non-singular $C_2$-invariant form.
	
	\begin{proposition}\label{WG=M}
		Let $p \geq 5$ prime, let $G$ be a finite rational $2$-group and let $V$ be a non-zero simple $\F_pG$-module with the eigenvector property. Then $p = 5$ and there is a subspace $W$ of $V$ such that if $N = N_G(W)$ and $C = C_G(W)$, then $W \rtimes N/C \cong \MM$ and $V\cong \Ind_N^G(W)$.
	\end{proposition} 
	\begin{proof}
		As $G$ is rational, $V$ is absolutely irreducible as $\F_pG$-module. Hence, the Brauer character $\chi$ of
		$G$ afforded by $V$ is irreducible and $\chi\in \Irr(G)$. By Lemma 2.5 of \cite{Hegedus2005}, $\chi = \varphi^G$ for a linear character $\varphi$ of a subgroup $H$ of $G$ with $\QQ(\varphi) \subseteq \QQ(i)$. Moreover, by the eigenvector property, $G$ has an element of order $p-1$, so that $p-1 = 2^k$ for some $k \geq 2$. This implies that $\F_p$ has a $4$-th root of unity, and hence, $\varphi$ is the Brauer character of an absolutely irreducible $\F_pH$-module $U$, as $\varphi$ is linear and $\QQ(\varphi) \subseteq \QQ(i)$. By Frobenius reciprocity, $\varphi$ is a constituent of $\chi_H$ and hence $U$ is a one-dimensional direct summand of $V_{\F_pH}$. Applying once more Frobenius reciprocity, it follows that $V$ is a constituent of $\Ind_H^G(U)$. But
		$\dim V = [G : H] = \dim \Ind_H^G(U)$ and hence $V = \Ind_H^G(U)$.
		
		Let $U = \F_pu$ and $C = C_G(u)$ for some $u\in U$. Observe that $H = N_G(U)$. We claim that $C$ is maximal in $\{C_G(v) : v \in V \setminus \{0\}\}$. Indeed, suppose that $C \subseteq D = C_G(v)$ with $v \in V \setminus \{0\}$ and let
		$U_1 = \F_pv$ and $H_1 = N_G(U_1)$. By the eigenvector property $[H : C] = [H_1 : D] = p - 1$. On the other hand $U_1$ is a simple submodule of $V_{\F_pH_1}$. Then, by Frobenius reciprocity, $V$ is a direct summand of $V_1 = \Ind_{H_1}^G(U_1)$. Therefore,
		$$[G : C] = (p - 1)[G : H] = (p - 1) \dim V \leq (p - 1) \dim V_1 = (p - 1)[G : H_1] = [G : D]$$
		so $C = D$. This finishes the proof of the claim.
		
		As $G$ is rational, $V$ is isomorphic to its contragredient and hence $V$ has a non-singular $G$-invariant
		form by Lemma 4 of \cite{Gow}. Let $W = C_V (C)$ and $N = N_G(W)$. By Lemma \ref{nonsingular}, the restriction of this
		form to $W$ is $N$-invariant and non-singular. Moreover, $W$ has the eigenvector property as $\F_pN$-module
		and hence it also has the eigenvector property as $\F_p(N/C)$-module. Also, $C = C_G(w)$ for
		every $w \in W \setminus \{0\}$, so that the action of $N/C$ on $W$ is fixed-point-free. By Proposition \ref{p=5}, $p = 5$,
		$N/C \cong Q_8$ and $W \rtimes N/C \cong \MM$. Then $W$ is the unique up-to-isomorphism simple $\F_5(N/C)$-module of
		dimension $2$, $H/C$ is a cyclic subgroup of $N/C$ of order $4$ and $U_{H/C}$ is a one dimensional summand of $W_{H/C}$. Then $W \cong \Ind_H^N(U)$ and hence $V \cong \Ind_H^G(U) = \Ind_N^G(W)$.
	\end{proof}
	
	Before we continue with the proof of the main theorem, let us recall something about module inductions.
	
	\begin{remark}\label{induction} \textup{Suppose that $V = \Ind_H^G(W)$ for $H$ a subgroup of $G$ and $W$ an $\F H$-module. Let $C = C_H(W)$ and let $\{t_1,\cdots, t_n\}$ be a right transversal of $H$ in $G$. Then every $g \in G$ defines a permutation $\sigma_g \in \Sigma_n$ so that $\sigma_g(i)$ is uniquely determined by $t_ig \in Ht_{\sigma_g(i)}$. Then $$g \longmapsto (t_igt^{-1}_{\sigma_g(i)}C)_{i=1}^n \sigma_g$$
	is a group homomorphism from $G$ to $(H/C) \wr \Sigma_n$ whose kernel is the core of $C$ in $G$, which coincides
	with the kernel of the action of $G$ on $V$. Therefore, if $V_G$ is faithful, then $G$ is isomorphic to a subgroup of $(H/C) \wr \Sigma_n$.}\end{remark}
	
	\begin{proof}[Proof of Theorem A] Suppose first that $V$ is simple. In that case, by Proposition \ref{WG=M}, $p = 5$ and $V = \Ind^G_N(W)$ with $W$ a subspace of $V$ such that $N = N_G(W)$ and if $C = C_G(W)$, then $N/C \cong Q_8$ and
	$W \rtimes N/C \cong \MM$. By Remark \ref{induction}, $G$ is isomorphic to a subgroup of $Q_8 \wr \Sigma_n$ and $2n = \dim V$, where
	$n = [G : N]$. As $G$ is a $2$-group we also have that $G$ is isomorphic to a subgroup of $Q_8 \wr S$ for a Sylow
	$2$-subgroup $S$ of $\Sigma_n$. Moreover, $V_G$ has the eigenvector property and hence, by Proposition 3.3 of \cite{Hegedus2005}, $G$
	is itself a wreath product $Q_8 \wr K$ for some $2$-subgroup $K$ of $\Sigma_n$. Now, using that $V = \Ind^G_N(W)$, it easily follows that $V \rtimes G \cong \MM\wr K$.
	
	For the general case, write $V =\bigoplus_{i=1}^m V_i$ with each $V_i$ a simple $\F_pG$-module. For every $i \in  \{1,\dots,m\}$,
	let $C_i = C_G(V_i)$. Then $V_i$ is a simple faithful $\F_p(G/C_i)$-module with the eigenvector property and
	hence, by the previous paragraph, $p = 5$, $G/C_i \cong Q_8 \wr K_i$ for some $2$-subgroup $K_i$ of $\Sigma_{n_i}$ with $2n_i = \dim V_i$ and $V_i\rtimes (G/C_i) \cong \MM\wr K_i$. Let $n = n_1+\cdots +n_m$. Then $\dim V = 2n$ and considering $K_1 \times \cdots \times K_m$ as a subgroup of $\Sigma_n$, via the natural embedding of $\Sigma_{n_1} \times \cdots \times \Sigma_{n_m}$ in $\Sigma_n$, we have
	$$\prod_{i=1}^m(V_i \rtimes G/C_i) \cong \prod_{i=1}^m \MM\wr K_i \cong \MM^n \rtimes \prod_{i=1}^m K_i \cong \MM \wr \prod_{i=1}^m K_i.$$

	Let $\alpha$ be this composition of isomorphisms from left to right. Then $\alpha(V) \cong C_{5}^{2n}$ is the unique Sylow
	$5$-subgroup of the image of $\alpha$, and $\alpha(\prod_{i=1}^m G/C_i)= Q_8 \wr \prod_{i=1}^m K_i$. Since $V_G$ is faithful, the natural
	map $\beta: G \longrightarrow \prod_{i=1}^m G/C_i$ is injective so that $G\cong \alpha \beta(G)\subseteq Q_8 \wr \prod_{i=1}^m K_i$. Applying  Proposition 3.3 of \cite{Hegedus2005} once more, we deduce that $G \cong \alpha \beta(G) = Q_8 \wr K$ for some $2$-subgroup $K$ of $\Sigma_n$. Thus $V \rtimes G \cong \alpha(V) \rtimes \alpha \beta(G) \cong \MM\wr K$.
	\end{proof}
\medskip
\textbf{Acknowledgment:} Theorem A came across in the middle of a different research about the prime graphs of solvable \cut groups \cite{DebonGLucasdelRio26}. The author is grateful to Ángel del Río and Diego García-Lucas for their encouragement to publish the result independently. Also for their useful feedback, comments
and support during the preparation of this paper. Moreover, the research is partially supported by grants PID2024-155576NB-I00 and 22004/PI/22 funded by MICIU/AEI/ 10.13039/501100011033 /FEDER, UE and Fundación Séneca, respectively.

\bibliographystyle{plain}
\bibliography{References}

\end{document}